%% file: main.tex
\documentclass[12pt]{amsart}
\usepackage{amsmath}
\usepackage{amssymb}
\usepackage{amsthm}

\newcommand{\ignore}[1]{\relax}

\numberwithin{equation}{section}

\newcommand{\C}{\mathbb C}

\newcommand{\R}{\mathbb R}
\newcommand{\Z}{\mathbb Z}

\newcommand{\T}{\mathbb T}

\newcommand{\Arg}{\operatorname{Arg}}

\newtheorem{lem}{Lemma}
\newtheorem{claim}{Claim}[section]

\newtheorem{theorem}[lem]{Theorem}
\newtheorem{corollary}[lem]{Corollary}
\newtheorem{thm}[lem]{Theorem}

\newtheorem{coro}[lem]{Corollary}
\newtheorem{prop}[lem]{Proposition}

\theoremstyle{definition}
\newtheorem{condition}[claim]{Condition}
\newtheorem{defn}[claim]{Definition}

\newtheorem{problem}[claim]{Problem}
\newtheorem{exa}[claim]{Example}

\theoremstyle{remark}
\newtheorem{rmk}[claim]{Remark}

\newcommand{\ctor}{(\C^\times)^{2n}}

\newcommand{\rtor}{(\R^\times)^{2n}}

\newcommand{\sonen}{(S^1)^{2n}}

\newcommand{\conj}{\operatorname{conj}}

\newcommand{\dd}{\partial}
\newcommand{\am}{\mathcal{A}}

\newcommand{\cp}{{\mathbb C}{\mathbb P}}
\newcommand{\rp}{{\mathbb R}{\mathbb P}}

\newcommand{\Log}{\operatorname{Log}}

\newcommand{\Vol}{\operatorname{Vol}}

\renewcommand{\setminus}{\smallsetminus}

\ignore{
\newtheorem{theorem}{Theorem}

\newtheorem{claim}[theorem]{Claim}

\newtheorem{condition}[theorem]{Condition}

\newtheorem{problem}[theorem]{Problem}

\newenvironment{proof}[1][Proof]{\noindent\textbf{#1.} }{\ \rule{0.5em}{0.5em}}
}

\newcommand{\coam}{\mathcal B}

\begin{document}
\title{Amoebas of half-dimensional varieties}
\author{Grigory Mikhalkin}
\address{Universit\'e de Gen\`eve,  Math\'ematiques, Villa Battelle, 1227 Carouge, Suisse}
\begin{abstract}
An $n$-dimensional algebraic variety in $\ctor$ covers its
amoeba as well as its coamoeba generically finite-to-one.
We provide an upper bound for the volume of these amoebas as 
well as for
the number of points in the inverse
images under the amoeba and coamoeba maps.  
\end{abstract}
\thanks{Research is supported in part by the grant TROPGEO of the European Research Council,
by the grants 140666 and 141329 of the Swiss National Science Foundation,
and by the NCCR SwissMAP of the Swiss National Science Foundation.}
\maketitle
\input{a12}

\input{cov}

\bibliography{b}
\bibliographystyle{plain}

\end{document}

%% file: a12.tex
\section{Introduction}
\subsection{Definitions}
Consider an $n$-dimensional 
algebraic variety $V\subset\ctor$.
\begin{defn}[Gelfand-Kapranov-Zelevinsky \cite{GKZ}]
The {\em amoeba} $\am$ of $V$
is the image $$\am=\Log(V)\subset\R^{2n}$$ of $V$
under the coordinatewise logarithm map $\Log:\ctor\to\R^{2n}$,
$$\Log(z_1,\dots,z_{2n})=(\log|z_1|,\dots,\log|z_{2n}|).$$
The restriction $\Log|_V$ is called the {\em amoeba map} for $V$.
\end{defn}
\begin{defn}[cf. Passare \cite{Pa-coam}]
The {\em coamoeba} (or {\em alga}, cf. \cite{FHKV})
$\coam$ of $V$
is the image $$\coam=\Arg(V)\subset\sonen$$ of $V$
under the coordinatewise argument map $\Arg:\ctor\to(\R/2\pi\Z)^{2n}\approx\sonen$,
$$\Arg(z_1,\dots,z_{2n})=(\arg(z_1),\dots,\arg(z_{2n})).$$
The restriction $\Arg|_V$ is called the {\em coamoeba map} for $V$.
\end{defn}
For coamoebas it is often more convenient to use
argument taken mod $\pi$ instead of mod $2\pi$,
(cf. \cite{MiOk}).
Namely, we denote $T_\pi=\R/\pi\Z$ and for
$z\in\C^\times$ we define
$$\arg_{\pi}(z)=(\arg(z)\hspace{-10pt}\mod\pi)\in\T_\pi.$$
In other words, $\arg_\pi$ is the composition of $\arg$ and 
the double covering $\R/2\pi\Z\to T_\pi$.
Then we define $\Arg_\pi:\ctor\to(\R/\pi\Z)^{2n}=T_\pi^{2n}$,
and
$$\Arg_\pi(z_1,\dots,z_{2n})=(\arg_\pi(z_1),\dots,\arg_\pi(z_{2n})).$$
We call $\coam_\pi=\Arg_\pi(V)\subset T^{2n}_\pi$ the
{\em rolled coamoeba} of $V$.   

We consider two antiholomorphic involutions on $\ctor$: 
$\conj,\conj':\ctor\to\ctor$ defined by
\begin{equation}\label{conj}
\conj(z_1,\dots,z_{2n})=(\bar z_1,\dots,\bar z_{2n}),\
\conj'(z_1,\dots,z_{2n})=(\frac1{\bar z_1},\dots,\frac1{\bar z_{2n}}).
\end{equation}
To each $n$-dimensional variety $V\subset\ctor$ we associate 
two integer numbers. Note that if $A,B\subset\ctor$ are two complex
subvarieties of complimentary dimensions then for an open dense
subset of $\epsilon\in\ctor$ all intersection points from
$A\cap\epsilon B$ are transverse and their number does not depend 
on $\epsilon$ (as long as it is generic).
Here $\epsilon B$ stands for the coordinatewise multiplication of $B$ by
$\epsilon$ in $\ctor$ 
(in other words for the multiplicative translation).
We define the toric intersection number $A.B\in\Z_{\ge 0}$
to be the number of points in $\#(A\cap\epsilon B)$ for a generic $\epsilon$
(times the corresponding multiplicities in the case when the corresponding
components of $A$ or $B$ are not simple. i.e. if $A$ or $B$ are not reduced).
Clearly, $A.B=B.A$.

\begin{defn}
We define the {\em $\conj$-degree}
$$\alpha(V)=V.\conj (V)\in\Z_{\ge 0}$$
and the {\em $\conj'$-degree}
$$\beta(V)=V.\conj' (V)\in\Z_{\ge 0}.$$
\end{defn}

\begin{defn} Let $A$ and $B$ be two smooth (differentiable)
manifolds of the same dimension and $f:A\to B$ be a smooth map.
We say that $f$ covers its image at most $m$ times if for any
point $p\in B$ which is regular for $f$
the inverse image $f^{-1}(p)$ consists of at most $m$ points.

More generally, if $A$ is a (not necessarily smooth)
real or complex algebraic variety (such as $V\subset\ctor$
in the case when it is singular),
it admits a stratification into smooth manifolds.
Consider a map $f:A\to B$ whose restriction $f|_\Sigma$
to every stratum $\Sigma\subset A$ is smooth.
Similarly,
we say that $f$ covers its image at most $m$ times if for any
point $p\in B$ the inverse image 
$f^{-1}(p)$ consists of at most $m$ points
unless $p$ is a critical point for $f|_\Sigma$,
where $\Sigma\subset A$ is a stratum of our stratification.
\end{defn}

\subsection{Statement of the results}
The main results of this paper are contained
in the following theorem.
\begin{thm}\label{mthm}
Let $V\subset\ctor$ be an algebraic $n$-dimensional variety.
Then the amoeba $\am(V)$ is covered by
the map $\Log|_V$ at most $\beta(V)$ times,
while the rolled coamoeba $\coam_\pi(V)$ is covered
by the map $\Arg_\pi|_V$ at most $\alpha(V)$ times.
Furthermore,
$$\Vol(\am)\le \frac{\pi^{2n}}{2}\alpha(V).$$
\end{thm}
Note that the conventional (i.e. non-rolled) coamoeba
$\coam(V)$ cannot be covered more than the rolled coamoeba.
Thus it is also covered by the coamoeba map at most $\alpha(V)$ times.

If $V$ is a complete intersection
then we can easily compute the $\conj$-degree
$\alpha$ as well as the $\conj'$-degree $\beta$ by means of the
Bernstein-Kouchnirenko calculus as follows.
\begin{defn}
We say that
$$V=\bigcap\limits_{j=1}^n V_j\subset\ctor$$
is a toric {\em complete intersection}
of hypersurfaces $V_1,\dots,V_n\in\ctor$
if 
$$V=\lim\limits_{\epsilon_j\to 0}
\bigcap_{j=1}^n\epsilon_j V_j.$$
Here the limit is taken in the sense of Hausdorff metric
on the subsets of $\ctor$ (with a group invariant metric)
and $\epsilon_j$. In particular, we require this limit to exist.
\end{defn}

\begin{prop}\label{c-int}
Suppose that $V=\bigcap\limits_{j=1}^n V_j\subset\ctor$
is a complete intersection of hypersurfaces $V_j$ with
Newton polyhedra $\Delta_j\subset\R^{2n}$, $j=1,\dots,n$.
Then we have
$$\alpha(V)=\Vol(\Delta_1,\dots,\Delta_n,\Delta_1,\dots,\Delta_n)$$
and 
$$\beta(V)=\Vol(-\Delta_1,\dots,-\Delta_n,\Delta_1,\dots,\Delta_n).
$$
Here $\Vol$ stands for the mixed volume of $2n$ polyhedra in $\R^{2n}$.
\end{prop}
\begin{rmk}
Proposition \ref{c-int} and Theorem
\ref{mthm} produce upper bounds for the volumes
of amoebas in the case of toric complete intersections
in terms of the mixed volumes of the corresponding Newton polyhedra.
Such bounds were conjectured
in the talk by Mounir Nisse on
the memorial conference for Mikael Passare in Summer 2013. 
Finiteness of $\Vol(\am)$ was observed in \cite{MaNi}.
\end{rmk}
\begin{proof}
Note that
$\conj(V)$ is a also a toric complete intersection defined by
the polynomials with the same Newton polyhedra,
but conjugate coefficients while $\conj'(V)$ is a toric
complete intersection
defined by the polynomials with $-\Delta_j$
as their Newton polyhedra as we need to make a substitution
$z_j\mapsto\frac1{z_j}$ before conjugation.
The proposition now follows from the Bezout theorem in the form
of Bernstein-Kouchnirenko \cite{Be}, \cite{Kou}.
\end{proof}
In particular, if $V$ is a toric complete intersection with
\begin{equation}\label{Delta1}
\Delta_1=\dots=\Delta_n=
\{(x_1,\dots,x_{2n})\in\R_{\ge 0}^{2n}\ |\
\sum\limits_{j=1}^{2n} x_j \le 1\}
\end{equation}
then $\alpha(V)=1$ and $\beta(V)=\frac{(2n)!}{(n!)^2}$ so 
Theorem \ref{mthm} has the following corollary.
\begin{coro}
If $V=\bar{V}\cap \ctor$, and $\bar V$ is a complete intersection
of hypersurfaces of degrees $d_1,\dots,d_n$ in $\cp^{2n}$ then
$\am(V)$ is covered by the amoeba map at most
$\frac{(2n)!}{(n!)^2}\prod\limits_{j=1}^nd_j^2$ while
$\coam_\pi(V)$ (as well as $\coam(V)$ itself) is covered at most
$\prod\limits_{j=1}^nd_j^2$ times by the coamoeba map.
Furthermore,
$$\Vol(\am)\le\frac{\pi^{2n}}{2}\prod\limits_{j=1}^nd_j^2.$$
\end{coro}


%% file: cov.tex
\section{Proof of the theorem}
\subsection{Bounds for the number of inverse images
for the amoeba and coamoeba maps}
In this section we prove the first part of 
Theorem \ref{mthm} establishing bounds for the 
number of inverse images of $\Log|_V$ and $\Arg_\pi|_V$.

Note that 
$(\Arg_{\pi}|_V)^{-1}(0)=V\cap (\R^\times)^{2n}\subset\ctor$.
We may compare this with $(\Arg|_V)^{-1}(0)=V\cap (\R_{>0})^{2n}$
in the case of the conventional (not rolled) coamoeba map. 
Similarly, for $p\in T_\pi^{2n}=(\R/\pi\Z)^{2n}$ we have
$$(\Arg_{\pi}|_V)^{-1}(p)=V\cap e^{ip} (\R^\times)^{2n}\subset\ctor$$
where $e^{ip}\in\ctor$ is obtained by coordinatewise 
exponentiating of $ip$. If $p$ is a regular value of $\Arg_\pi|_V$
then $V$ intersects $e^{ip} \rtor$ transversally.
This means that every
stratum in a stratification of $V$ into smooth manifolds
intersects $e^{ip} \rtor$ transversally. By the dimension
considerations, the top-dimensional stratum intersects 
$e^{ip} \rtor$ in finitely many points while smaller-dimensional
strata are disjoint from $e^{ip} \rtor$.

Furthermore, we have the inclusion
\begin{equation}
\label{r-inclusion}
(\Arg_\pi|_V)^{-1}(p)=V\cap e^{ip} \rtor\subset V\cap e^{2ip}\conj(V)
\end{equation}
as $e^{ip} \rtor\subset\ctor$ is the invariant locus 
for the complex conjugation
$$(z_1,\dots,z_{2n})\mapsto e^{2ip}(z_1,\dots,z_{2n})$$
in $\ctor$.
Thus the cardinality of $(\Arg_\pi|_V)^{-1}(p)$ for regular $p$
is bounded by $\alpha(V)$ as stated in Theorem 1.

Similarly, for a regular value $q\in\R^{2n}$ of $\Log|_V$
we have a finite number of points in $(\Log|_V)^{-1}(q)$ as well as
the inclusion
\begin{equation}
\label{a-inclusion}
(\Log|_V)^{-1}(q)=V\cap e^{q}S\subset V\cap e^{2q}\conj'(V),
\end{equation}
where $S\subset\ctor$ is the unit torus (the fixed point locus of $\conj'$).
Thus the cardinality of $(\Log|_V)^{-1}(q)$ for regular $p$
is bounded by $\beta(V)$ as stated in Theorem 1. 

\subsection{Estimating the volume of amoeba}
To finish the proof of Theorem \ref{mthm}
we consider the real-valued $2n$-form on $\ctor$
\begin{equation}
\omega=\prod\limits_{j=1}^{2n}dx_j - \prod\limits_{j=1}^{2n}dy_j.
\end{equation}
Here product stands for the exterior product of
differential forms. 
\begin{lem}
\label{lem-vol}
We have
$\omega|_V\equiv 0$.
\end{lem}
\begin{proof}
We may write
\begin{equation}\label{omega2n}
\omega=\frac{1}{2^{2n}}(\prod\limits_{j=1}^{2n}(dz_j+d\bar z_j)
- (-1)^n \prod\limits_{j=1}^{2n}(d z_j-d\bar z_j)).
\end{equation}
The right-hand side of this expression 
is the sum of monomials of degree $2n$ in $dz_j$ and $d\bar z_k$.
Note that
if $n$ is odd then
there are no monomials with odd number of $d\bar z_j$.
Similarly, if $n$ is even then 
there are no monomials with even number of $d\bar z_j$.
Thus the right-hand side
of \eqref{omega2n} contains only monomials where
either the number of $dz_j$ is more than $n$
or the number of $d\bar z_k$ is more than $n$.
Thus, $\omega$ must vanish everywhere on a holomorphic $n$-variety $V$.
%
\end{proof}

We may consider the cardinality $\#(Log|_V)^{-1}$
of the inverse image of the amoeba map as
a measurable function on $\R^{2n}$ (since the critical locus
of $\Log|_V$ is nowhere dense).
Then $$\operatorname{MultiVol}(\am)=\int\limits_{\R^{2n}}\#((Log|_V)^{-1}(x_1,\dots,x_{2n}))
dx_1\dots dx_{2n}$$ can be thought of as the volume of $\am$
taken with the multiplicities corresponding to the covering
by the amoeba map. Similarly, 
$$\operatorname{MultiVol}(\coam_\pi)=
\int\limits_{T^{2n}_\pi}\#((Arg_\pi|_V)^{-1}(y_1,\dots,y_{2n}))
dy_1\dots dy_{2n}$$
can be thought of as the volume of $\coam_\pi$ taken
with the multiplicities corresponding to the covering
by the coamoeba map.
\begin{coro}\label{multi-eq}
$\operatorname{MultiVol}(\am)=\operatorname{MultiVol}(\coam_\pi)$.
\end{coro}
\begin{proof}
Let $V_+\subset V$ be the open subset of $V$ where
 the real $2n$-form $dx_1\wedge\dots\wedge d x_{2n}$ is non-degenerate 
and defines the orientation
that agrees with
the complex orientation of $V$. Let $V_-\subset V$
be the open set where these orientations disagree. 
Note that by Lemma \ref{lem-vol} the form
$dy_1\wedge\dots\wedge dy_{2n}$ also agrees with the complex
orientation on $V_+$ and disagrees on $V_-$.
We have 
$$\operatorname{MultiVol}(\am)=
\int\limits_{V_+}dx_1\wedge\dots\wedge d x_{2n}-
\int\limits_{V_-}dx_1\wedge\dots\wedge d x_{2n}
$$
while
$$\operatorname{MultiVol}(\coam_\pi)=
\int\limits_{V_+}dy_1\wedge\dots\wedge d y_{2n}-
\int\limits_{V_-}dy_1\wedge\dots\wedge d y_{2n}.
$$
The two multivolumes are equal by Lemma \ref{lem-vol}.
\end{proof}

\begin{lem}\label{multi-est}
$$\operatorname{MultiVol}(\am)=
\operatorname{MultiVol}(\coam_\pi)\le \alpha(V)\pi^{2n}.$$
\end{lem}
\begin{proof}
By \eqref{r-inclusion} the cardinality of $(\Arg_\pi|_V)^{-1}(p)$
is not greater than $\alpha(V)$
almost everywhere on $T_\pi^{2n}$ while $\Vol(T_\pi^{2n})=\pi^{2n}$. 
\end{proof}

Note that $\Log:\ctor\to\R^{2n}$ is a proper map (inverse images
of compact sets are compact) and thus $\Log|_V:V\to\R^{2n}$
is also proper.
Since $V$ and $\R^{2n}$ are oriented
manifolds of the same dimension
the map $\Log|_V$ has a well-defined degree. 
Recall that this degree is equal to the number of 
inverse images of a generic point $q\in\R^{2n}$
taken with the sign $\pm1$ depending whether
$\Log|_V$ locally preserves the orientation. 
\begin{coro}\label{deg0}
The degree of the amoeba map is zero.
We have $$\operatorname{MultiVol}(\am)=
2\int\limits_{V_+}dx_1\wedge\dots\wedge d x_{2n}=
-2\int\limits_{V_-}dx_1\wedge\dots\wedge dx_{2n}.$$
Furthermore, $\Vol(\am)\le
\frac12 \operatorname{MultiVol}(\am)$.
\end{coro}
\begin{proof}
As the multivolume of $V$ is bounded by Lemma \ref{multi-est}
we have $\R^{2n}\setminus\am\neq\emptyset$
and thus the degree of $\Log|_V$ must be zero.
Each generic point $q\in\am$
is covered by $V_+$ and $V_-$ the same number of times.
\end{proof}

\begin{rmk}\label{alpha-parity}
Note that Corollary \ref{deg0} immediately implies
that $\beta(V)$ is always even as it coincides
with the degree of the amoeba map $\Log|_V$
(this fact is also easy to deduce from symmetry reasons).
However, as the maps $\Arg|_V$ and $\Arg_\pi|_V$
are not proper, we cannot apply the same reasoning.
Note, in particular, that the parity of $(\Arg|_V)^{-1}(p)$
is different for different generic points of $(\R/2\pi\Z)^{2n}$
already in the case when $V\subset\cp^2$
is a generic line (cf. e.g. \cite{Mi-am}).

In the same time, \eqref{r-inclusion} implies 
that for 
the parity of $(\Arg_\pi|_V)^{-1}(p)$
coincides with $\alpha(V)$ for generic points $p\in T_\pi^{2n}$
as non-real intersection points of $e^{-ip}V$ and $e^{ip}\conj(V)$
come in pairs.
Thus
the rolled coamoeba map 
has a well-defined degree mod 2 determined by $\alpha(V)$.
\end{rmk}

Corollary \ref{deg0} implies that 
$$\Vol(\am)=\Vol(\Log(V_+))\le
\frac12\operatorname{MultiVol}(\coam_\pi)\le
\frac{\pi^{2n}}{2}\alpha(V).$$
This finishes the proof of Theorem \ref{mthm}.

\section{Some remarks and open problems}
\subsection{Example: linear spaces in $\cp^{2n}$}
Let $L\subset\cp^{2n}$ be an $n$-dimensional linear 
subspace that is generic with respect to the coordinate
hyperplanes of $\cp^{2n}$. Then $V=L\cap\ctor$
can be presented as a complete intersection of 
hyperplanes with the Newton polyhedra given by \eqref{Delta1}.
By Proposition \ref{c-int} we have $\alpha(V)=1$.
By Remark \ref{alpha-parity} the set
$(\Arg_\pi|_V)^{-1}(p)$ must consist of a single point
for almost all values $p\in T_\pi^{2n}$,
so the inequality of Lemma \ref{multi-est} turns
into equality.
We get the following proposition.
\begin{prop}[cf. \cite{Jo}, \cite{NiPa}]
\label{linear-mult}
If $V=\bigcap\limits_{j=1}^n\subset\ctor$ is a transverse intersection
of $n$ hyperplanes
$$H_j=\{(z_1,\dots,z_{2n})\ |\ a_{j0}+
\sum\limits_{k=1}^{2n}a_{jk}z_k=0\}$$
with $\prod\limits_{j=1}^n\prod\limits_{k=0}^{2n}a_{jk}\neq 0$
then
$$\operatorname{MultiVol}(\am)=
\operatorname{MultiVol}(\coam_\pi)=
\Vol(\coam_\pi)=\pi^{2n}.$$
\end{prop}

In the case of $n=1$ we have $\beta(V)=2$.
By Corollary \ref{deg0} we have
$$\Vol(\am)=\frac12\operatorname{MultiVol}(\am)=\frac{\pi^2}2$$
in this case as in \cite{PaRu}. This equality was used by
Passare \cite{Pa-zeta} to give a new proof of Euler's formula
$\zeta(2)=\frac{\pi^2}{6}$.
In the case $n>1$ we have $\beta(V)>2$,
so Proposition \ref{linear-mult} only implies
the inequalities
\begin{equation}
\frac{\pi^{2n}(n!)^2}{(2n!)^2}\le
\Vol(\am)\le\frac{\pi^{2n}}2,
\end{equation}
and $\Vol(\am)$ might vary with $V$. 
Note that our linear subspace $V\subset\ctor$
varies in a $(n^2-n)$-dimensional family if we identify
subspaces that can be obtained from each other
by multiplication by $\epsilon\in\ctor$ (such multiplication
corresponds to a translation of amoeba and thus
does not change its shape or its volume).
\begin{problem}
What are the maximal and minimal possible values 
of $\Vol(\am)$? It would be interesting to solve
this problem already for $n=2$.
\end{problem}

\subsection{MultiHarnack varieties in $\rp^{2n}$}
\begin{defn}\label{mHarnack}
We say that an $n$-dimensional variety $V\subset\ctor$
is {\em multiHarnack} if 
$$\operatorname{MultiVol}(\am)={\pi^{2n}}\alpha(V).$$
\end{defn}

Let us recall the notion of simple Harnack curves
in $(\C^\times)^2$ (introduced in \cite{Mi00}).
According to the maximal volume characterization
given in \cite{MiRu} a curve $V\subset (\C^\times)^2$
can be presented as $V=\epsilon C$ for a simple Harnack
curve $C\subset (\C^\times)^2$ and a multiplicative 
vector $\epsilon\in (\C^\times)^2$ if and only if
we have $\Vol(\am)=\frac{\pi^2}2\alpha(V)$.

This class of curves was generalized to 
a larger class of curves in $(\C^\times)^2$,
(also called { multiHarnack curves})
by Lionel Lang (\cite{Lang}, \cite{Lang15}).
Definition \ref{mHarnack} gives the multiHarnack curves
in the case $n=1$.

According to Proposition \ref{linear-mult} all generic
linear spaces in $\cp^{2n}$ are multiHarnack.
\begin{problem}
Do there exist multiHarnack varieties of higher degree?
\end{problem}
Note that once $n>1$ being multiHarnack no longer 
implies being real even after multiplication by $\epsilon\in\ctor$
already for linear spaces.

\begin{rmk}
It might be instructive to compare Definition \ref{mHarnack}
against another attempt to generalize the definition of
simple Harnack curves from \cite{Mi00} to higher dimensions.
The survey \cite{Mi-am} gave a definition of torically 
maximal hypersurfaces of dimension $n$ generalizing
the Definition from \cite{Mi00} for $n=1$. 
However, it was recently shown (see \cite{BrMiRi})
that all torically maximal hypersurfaces in $\rp^{n+1}$
for $n>1$ have degree 1.
\end{rmk}

\subsection{Foliation of $\dd\am$}
\newcommand{\Crit}{\operatorname{Crit}}
In this subsection we suppose for simplicity that
$V\subset\ctor$ is smooth (otherwise we may restrict
ourselves to the smooth part of $V$). 
Let us look at the critical locus $C\subset V$
of the map $\Log|_V$
and its image $D=\Log(C)
\subset\am\subset\R^{2n}$
(also called the {\em discriminant locus}).
We have the following generalization of Lemma 3 from \cite{Mi00}.
\begin{prop}
The set $C$ consists of the points $z\in V$
where $V$ and $z\rtor$ are tangent.
\end{prop} 
As usual, $z\rtor$ stands for the coordinatewise
multiplication of $\rtor$ by $z\in\ctor$.
\begin{proof}
We have $z\in C$ iff there are vectors in $T_z V$
tangent to the argument torus $\Log^{-1}(q)$, where $q=\Log(z)$.
Any such vector multiplied by $i$ gives a vector tangent
both to $V$ and $z\rtor$, and vice versa.
\end{proof}
\begin{defn}
Let $z\in C$. Denote
$$F(z)=T_z(V)\cap T_z(z\rtor)\subset T_z(\ctor).$$
It is a real vector subspace of the tangent space $T_z(\ctor)$.
The {\em rank} of $z\in C$ is $\dim_{\R} F(p)$.
\end{defn}
We denote with $C_r\subset C$ the locus of critical points
of rank at least $r$. 
The following proposition follows immediately 
from the injectivity of $d\Log$ on the tangent space
to $z\rtor$.
\begin{prop}
The subspace $$(d\Log)(F(z))\subset T_q(\R^{2n})$$
has dimension $r$ for $z\in C_r$, $q=\Log(z)$. 
\end{prop}
Thus we get a preferred $r$-dimensional subspace in the tangent
space of $\Log(z)$ for each $z\in C_r$.

Let us choose a stratification of the discriminant locus $D$ 
to $k$-dimensional (non-closed) subvarieties $\Sigma_k$,
$$D=\bigcup\limits_{k=0}^{2n-1}\Sigma_k.$$
By a $k$-dimensional multidistribution on 
an open set $U$ of a manifold we mean specifying
a finite set of $k$-dimensional
subspaces of $T_q U$ for every $q\in U$ so that they
depend on $q$ smoothly.
\begin{lem}
For a generic point $q$ of $\Sigma_{2n-k}$ the set 
$\Log^{-1}(q)\cap C_{k}$ is finite and disjoint from $C_{k+1}$.
Furthermore,
for each point $z\in\Log^{-1}(q)\cap C_{k}$
the $k$-dimensional space $(d\Log)(F(z))$ is tangent
to $\Sigma_{2n-k}$. We have a $k$-dimensional
multidistribution (perhaps empty) of
an open dense subset of $\Sigma_{2n-k}$. 
\end{lem}
\begin{proof}
Since $\dim\Sigma_k=k$ and the rank of $d(\Log|_{C_{k+1}})$ is at most $2n-k-1$,
the image $\Log(C_{k+1})$ is nowhere dense in $\Sigma_k$.
Also the critical values of $\Log|_{C_k\cap\Log^{-1}(\Sigma_k)}$
(treated as a map from any of its smooth stratum to the open manifold $\Sigma_{2n-k}$)
are nowhere dense in $\Sigma_k$.
If $(d\Log)(F(z))$ is not tangent to $\Sigma_k\ni\Log(F(z))$ at a regular
point of $\Log|_{C_k}$ then the
rank of $d(\Log|_V)$ is at least $2n-k+1$ which contradicts to
the definition of $C_k$. 
\end{proof}


%

Suppose that $\am\subset\R^{2n}$ is non-degenerate,
i.e. the interior of $\am$ is non-empty.
(This condition is equivalent to the condition $C\neq V$,
i.e. to the condition that $\Log|_V$ has a regular point.)
Then the amoeba boundary $\dd\am$ is a $(2n-1)$-dimensional
subset of $D$.
Let us denote with $\dd_1\am$
the subset of $\dd\am$ formed by points $q$
such that $\Log^{-1}(q)\cap C$ consists of a single point.
\begin{coro}
We have a 1-dimensional non-empty multifoliation on
an open dense set in $\dd\am$.
This is a genuine 1-dimensional foliation on
an open dense set in $\dd_1\am$.
\end{coro}


 

\begin{rmk}
Since $V$ is $n$-dimensional (over $\C$) and $z\rtor$
is $2n$-dimensional (over $\R$) and totally real,
the maximal dimension of $F(z)$ is $n$.

Suppose that $V=\conj V$, i.e. $V$ is defined over $\R$.
Then we have $\R V=V\cap\rtor\subset C_{n}$. 
\end{rmk}

\subsection{Dimensions greater that half} 
Suppose that $V$ is
a $k$-dimensional algebraic variety in $(\C^\times)^n$
with $k>\frac n2$.
Then the generic fibers of $\Log|_V$ and $\Arg_\pi|_V$
are $n-2k$-dimensional varieties as $V$ is $2k$-dimensional
(over $\R$) and $\R^n$ is $n$-dimensional.
We may still present 
generic fibers  of $\Arg_\pi|_V$
and $\Log|_V$
as real algebraic varieties in a way similar to
the half-dimensional case (where those fibers were points).
For $p\in T^n_\pi$ and $q\in\R^n$
we consider the antiholomorphic involutions
$\conj_p,\conj'_q:(\C^\times)^n\to (\C^\times)^n$ defined by
\begin{equation}\label{conjp}
\conj_p(z)=e^{ip}(\conj(e^{-ip}z)),\
\conj'_q(z)=e^{q}(\conj'(e^{-q}z)),
\end{equation}
where $\conj$ and $\conj'$ are defined 
as in \eqref{conj}:
$$
\conj(z_1,\dots,z_{n})=(\bar z_1,\dots,\bar z_{n}),\
\conj'(z_1,\dots,z_{n})=(\frac1{\bar z_1},\dots,\frac1{\bar z_{n}}).
$$
Note that the fixed point set of $\conj_p$ is
$\Arg^{-1}_\pi(p)=e^{ip}(\R^\times)^n$ while
the fixed point set of $\conj'_q$ is 
$\Log^{-1}(q)$.
The following proposition is straightforward.
\begin{prop}
The antiholomorphic involutions $\conj_p$, $\conj_q$
act on algebraic varieties $V\cap\conj_p(V)$
and $V\cap\conj'_q(V)$ so that
the fixed point sets are $(\Arg_\pi|_V)^{-1}(p)$ and
$(\Log|_V)^{-1}(q)$.

Thus we may think of the fibers of $\Arg_\pi|_V$ and
$\Log|_V$ as real algebraic varieties whose complexification
is $V\cap\conj_p(V)$ and $V\cap\conj'_q(V)$. 
For regular fibers these varieties are non-singular 
$(n-2k)$-dimensional varieties near their real points.
\end{prop}

\begin{exa}
Consider the plane 
\begin{equation}\label{plane}
V=\{(x,y,z)\in (\C^\times)^3\ |\ 1+x+y+z=0\}
\end{equation}
in $(\C^\times)^3$.
Both $V$ and $\conj_p(V)$ are planes, so fibers
of $\Arg_\pi|_V$ are intersections of two real planes
in $(\R^\times)^3$ after a multiplicative translation by $e^{ip}$.
For generic fibers these two planes are transversal,
so their intersection is a line.
For special fibers these planes might be parallel planes,
or two copies of the same plane. These special cases correspond to
empty or two-dimensional fibers of $\Arg\pi|_V$.

The surface $\conj'_q(V)$ is the image of a plane
under the Cremona transformation $x\mapsto \frac 1x$,
$y\mapsto\frac 1y$, $z\mapsto\frac 1z$.
For a generic $q$ the intersection of $V$ and $\conj'_q(V)$
is a smooth elliptic curve. Its real locus may be empty,
or consist of one or two circles. All three cases are realized
as generic fibers of $\Log|_V$ for $V$ given by \eqref{plane}.

\end{exa}